\pdfoutput=1
\documentclass{article}
\usepackage[margin=25mm]{geometry}
\usepackage{amsmath}
\usepackage{amsfonts}
\usepackage{amssymb}
\usepackage{graphicx}

\usepackage{xcolor}
\usepackage[all]{xy}
\newtheorem{lemma}{Lemma}
\newtheorem{theorem}{Theorem}
\newtheorem{proof}{Proof}

\title{Topological structure of optimal flows  on the  Girl's surface}
\author{Maria Loseva and Alexandr Prishlyak}

\begin{document}
\maketitle

\begin{abstract}
We investigate the topological structure of flows on the  Girl's surfaces which is one of two possible  immersions of the projective plane in three-dimensional space with one triple point of the selfintersection.First, we describe the cellular structure of the  Boy's and Girl's surfaces and prove that there are unique images of the project plane in the form of a 2-disc, in which the opposite points of the boundary are identified and this boundary belongs to the preimage  of the 1-skeleton of the surface. Second, we described three structures of flows with one fixed point and no separatrix on the  Girl's surface and proved that there are no other such flows. Third, we proved that Morse-Smale flows and they alone are structurally stable on the Boy's and Girl's surfaces. Fourth, we have found all possible structures of optimal Morse-Smale flows on the Girl's surface. Fifth, we have obtained a classification of Morse-Smale flows on the projective plane, that immersed on the Girl's surface. And finally, we described the components of linear connectivity of sets of these flows.
\end{abstract} \hspace{10pt}



\section*{Introduction}

 The Boy's \cite{Boy03} and Girl's \cite{Goodman09, bridges2013:383} surfaces are immersions of the real projective plane in three-dimensional space. They are unique possible immersions  with a single triple point and connected set of self-intersections.  
 
We consider the Boy's and Girl's surfaces as a stratified sets. A Morse-Smale flow is one of the main instrument in the Morse theory on manifolds and stratified sets. 

A smooth flow on a closed manifold is always generated by a smooth vector field, but on a manifold with boundary (a stratified set) vector field have to be tangent to the boundary (every strata). Then
all the concepts that are introduced for dynamical systems (flow) can be applied to vector fields. Let $X_t$ be the flow generated by a vector field $X$.

Two vector fields (flows) $X,Y$ on $M$ are called topologically equivalent if there exists a homeomorphism $h: M\to M$ that sends each trajectory of $X$ onto a trajectory of $Y$ preserving their orientations.

 A flow $f^t$ on a closed manifold is called structurally stable if, for any sufficiently close flow $g^t$, there exists a homeomorphism $h$ sending orbits of the system $g^t$ to orbits of the system $f^t$ ( \cite{Peixoto59}).

 Fixed (singular) points and closed trajectories of a dynamical system are called critical elements.

A singular point $p$ of a vector field $X=\{X_1;\dots X_n\}$ is called non-degenerate or hyperbolic, if in some local coordinates $(x_1;\dots ;x_n)$ at $p $ the Jacobi matrix $\left( \frac{\partial X_i}{\partial x_j} \right) _{i;j=1}^n$ does not contain eigenvalues with zero real part. A closed orbit of $X$ is called hyperbolic whenever the differential of the Poincare map has no eigenvalues of absolute value1at some point $p_2$.

 A stable manifold  $S(p)$ of a criticalel ement $p$ is the following set $S(p) =\{ x\in M: \text{lim}_{t\to \infty }X_t(x) =p\}$. Similarly, the set $U(p) =\{ x\in M: \text{lim}_{t\to - \infty }X_t(x) =p\}$ is called the unstable manifold of $p$.

 A point $p\in M$ is wandering for $X$ if there exists a neighborhood $V$ of $p$ and a number $n >0$ such that $X_t(V)\setminus V= \emptyset$ for any $j> n$. 

A fixed point $y\in M$ is called an $\alpha$-limit(resp. an$\omega$-limit) point for $x\in M$ if there exists a sequence $\{t_i\}$ such that $t_i \to + \infty$ (resp.$t_i \to - \infty$) and $X_{t_i}(x)\to y$.

A flow is a \textit{Morse-Smale} flow if 1) its non-wandering set consists of finitely many singular points and periodic orbits, 2) each of which is hyperbolic, and 3) their stable and unstable manifolds intersect transversally (\cite{Smale60}).

The set of all $\alpha$-limit points (resp.$\omega$-limit points) of $x$ is called $\alpha$-limit ($\omega$-limit) set. Condition 1) of the definition of Morse-Smale dynamical system can be replaced by the following condition 1)' for each point $x\in M$ its $\alpha$-limit and $\omega$-limit sets are contained in the union of critical elements. For vector fields on two-dimensional manifolds there are three types of non-degenerated (hyperbolic) singular points: sinks, sources and saddles. Condition 3) in this case is equivalent to the statement that there is no trajectories whose $\alpha$-limit and $\omega$-limit sets are saddle points.

Morse-Smale flow (vector field) is called Morse flow (vector field) if it does not contain closed trajectories. 

There are a lot of papers on structural classification of Morse-Smale vector fields on surfaces. Most known are  \cite{Fleitas75, Kruglov2018, Leontovich55, Oshemkov98, Peixoto73}.

A Morse-Smale flow without closed orbit is called a Morse flow. We say that flow is optimal if it has the lowest number of fixed points among all flow of that type on the surface.
The Morse flow on the closed surface is optimal if and only if it has only one sink and one source ( \cite{Kibalko18}). Such a flow is also called a polar Morse flow. The topological structure of polar (optimal) Morse flows on closed 2- and 3-manifolds was described in \cite{Giryk96, kad05,
Kibalko18, maks11, Poltavec1995,  prish07ct, prish02top, prish02ms, PrishLos2020}.

For flows on a stratified set, their restriction to each stratum has the same structure as flows on a surface with a boundary. This situation has been investigated in \cite{Loseva2016,  Prus17, PPG2021, prish03sum, Prishlyak2019a, prish03tc, Prishlyak2019}.  Structural stability of Morse-Smale flows was proved for closed manifolds in \cite{ Palis1970, Smale61} and for manifolds with boundary in \cite{Percell_1973, Robinson_1980}.

Morse fields are topologically equivalent to gradient fields of Morse functions \cite{Smale61}. Moreover, the structure of the Morse field coincides with the structure of the Morse function, in which all points of the same index have equal function values. The structure of Morse functions on manifolds of dimensions 2 and 3 is described in \cite{Reeb1946, Kronrod1950, prish00, prish02te,prish02mf, prish03pr, Bolsinov2004, prish08, lychak-prish09, Hladysh-prish16, Hladysh-prish17, HladPrish2019,Hladysh_2019, Hlad-Prish2020}

The formula for the sum of fixed point indices is useful for calculating their number for flow on a stratified set \cite{prish03sum}.

In \cite{Dibeo-prish}, it was described the structure of such optimal flows on the Boy's surface: 1) flows with one fixed point, 2) ms-flows, 3) Morse-Smale projective flows.

The purpose of our article is to describe the structure of such fields on the Girl's surface, as well as to study these flows with respect to homotopy, that is, to describe the linearly connected components of the set of such flows.
The flow invariants that are constructed are graphs embedded in the surface. To encode them, one can use a rotation system as in topological graph theory or a graph with a list of words as in \cite{Prish97}.

Our paper has the following structure. In the first section, we describe the cellular structure of the Boy and Girl surfaces and prove that there are unique representations of the projective plane in the form of a 2-disk, which has opposite boundary points identified and this boundary belongs to the inverse image of the 1-skeleton of the surface.
In the second section, we find three flow structures with one fixed point and no separatrices on the Girl's surface and prove that there are no other such flows.
In the third section, it is proved that Morse-Smale flows and only they are structurally stable on Boy and Girl surfaces.
In the fourth section, all possible structures of optimal Morse-Smale flows on the Girl's surface are described. In the fifth section, we construct a classification of optimal Morse-Smale projective flows on the Girl's surface. Finally, in the last section, for each of the Boy's and Girl's surfaces, we prove formulas relating the numbers of topologically non-equivalent flows, symmetric flows, and non-isotopic flows.





\section{Natural sructures of CW-complexes on the Boy's and Girl's surfaces}

 Here we describe the natural CW-structures of the Boy's and Girl's surfaces, as well as the resulting CW-structures  of the projective plane. In this case, we use a model of the projective plane in the form of a 2-disk, in which opposite points on the boundary are identified (glued). The corresponding CW-structures are planar model of these surfaces. 
 
 \begin{figure}[ht]
\center{\includegraphics[height=5.0cm]{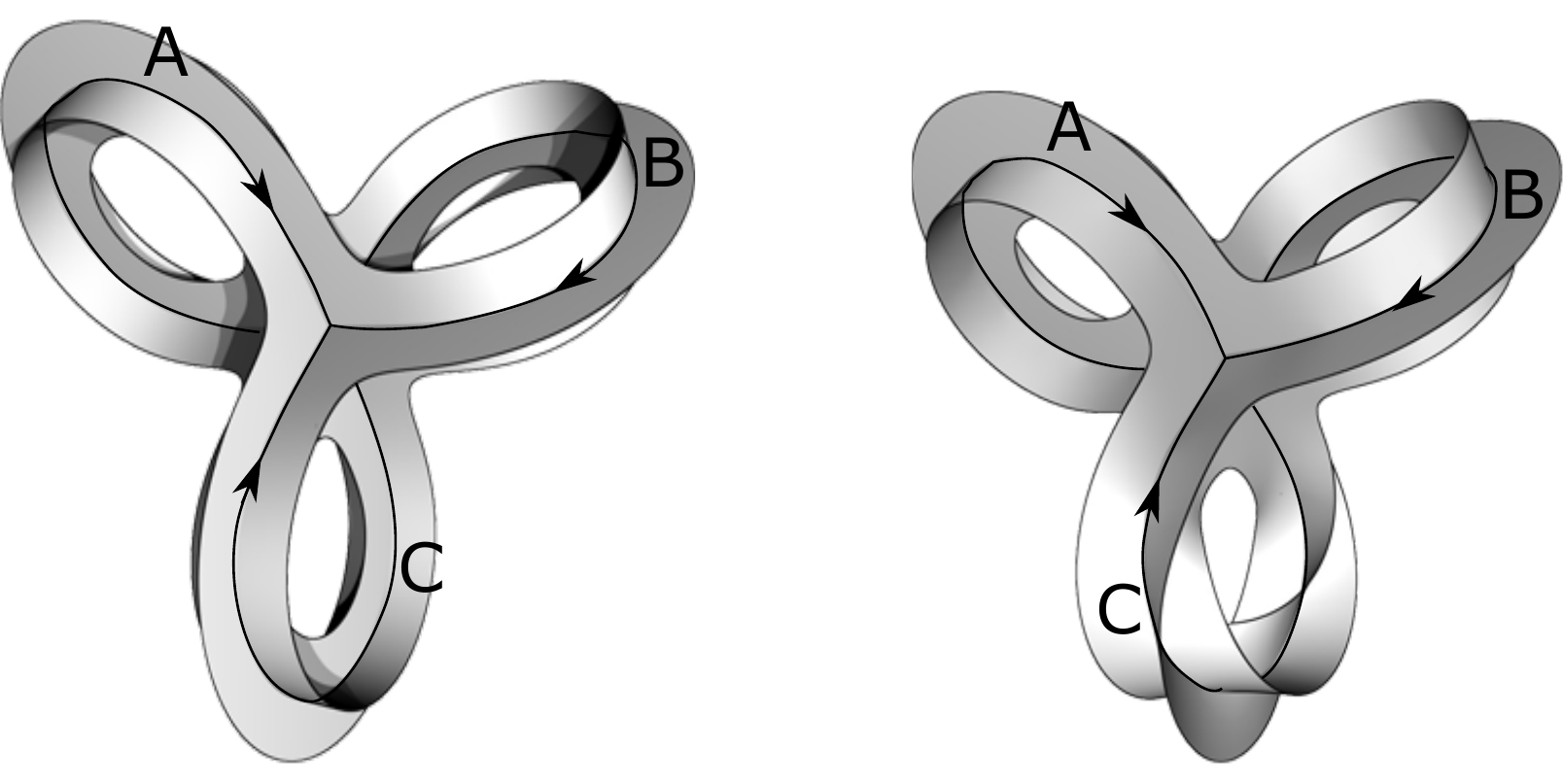}}
\caption{Neighborhoods of the selfintersections on the Boy's (left) and Girl's (right) surfaces}
\label{rp2}
\end{figure}

 On each surface there is one triple dot, which we  call null-point, and three loops formed by double points. Let's denote them by $A, B, C$.  These loops lie at the coordinate angles of different planes and touch the coordinate axes at their common point. We set the direction of movement along them so that the corresponding speed vectors in 0 coincide when moving from C to B, from B to A and from A to C.  The structure of CW-complex for each surface consists of a 0-cell -- null-point, three 1-cell -- $A, B, C$ and four 2-cells. The neighborhoods of the 1-skeleton are shown in Fig. \ref{rp2}. They determine the gluing of 2-cells.

We number the angles in null-point of this neighborhood with integers from 1 to 12 as in Fig. \ref{an}. 

\begin{figure}[ht]
\center{\includegraphics[height=5.0cm]{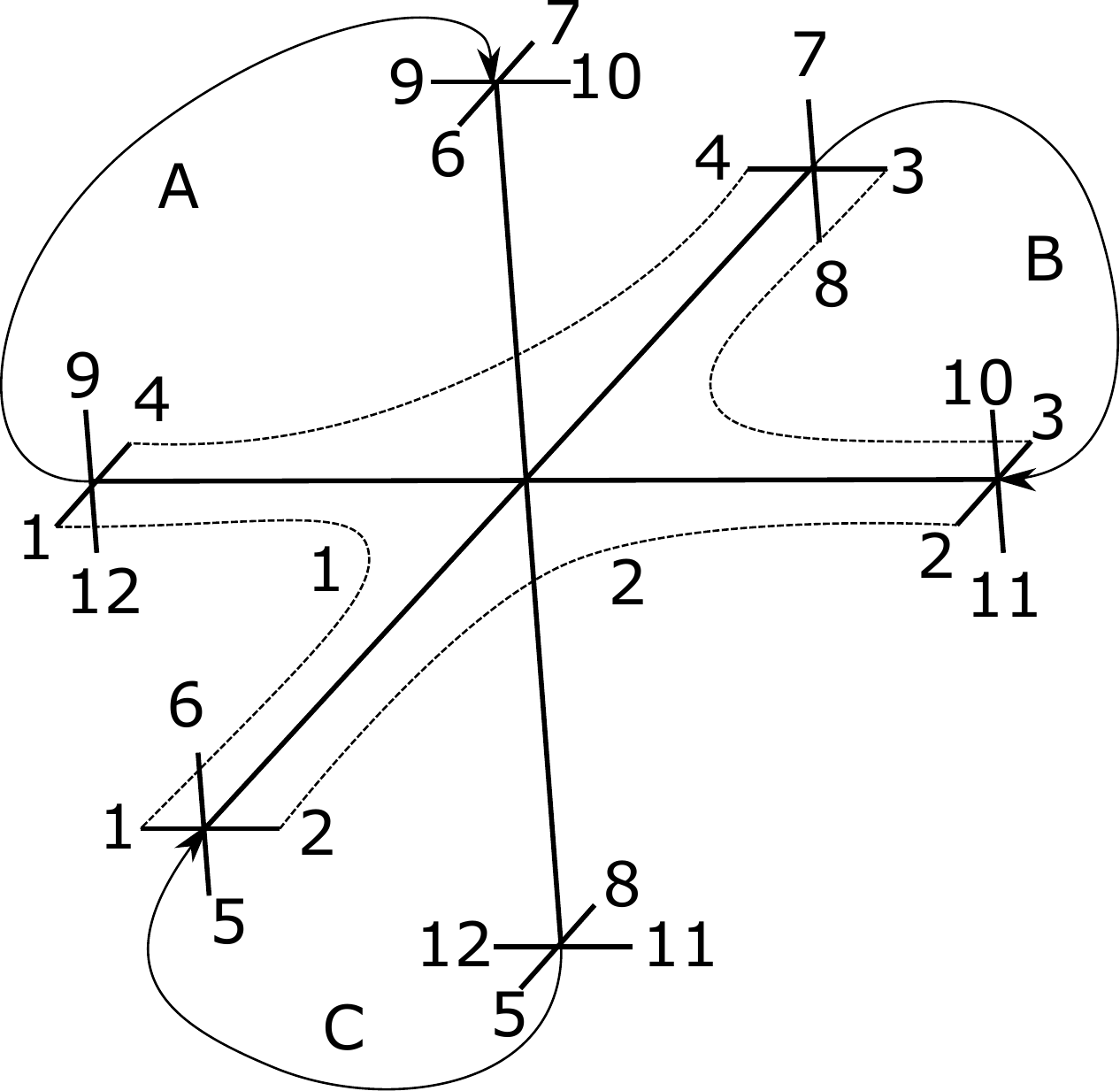}}
\caption{Angle numbering 
}
\label{an}
\end{figure}

The gluing of each 2-cell is determined by the sequence of angles and 1-cells that occur when bypassing its boundary. So in the Boy's surface there are three small 2-cells that lie in the coordinate angles and are limited by 1-cells: $9A'9, 3B'3, 5C'5$.  One large 2-cell has boundary $1A6C'8B11C2B'\-4A7\-B\-10\-A'12C1$.  After gluing three small cells to this nonagon to the hatched sides, we get a hexagon as in Fig. \ref{bg} on the left.

On the Girl's surface, the 1-cells are connected by the 
angles: $A:$ $9-9,1-6,12-10,4-7$, $B:$ $3-3, 7-10, 4-2, 8-11$, $C:$ $5-6, 12-2, 8-5, 11-1$. We get the following boundaries of 2-cells: $9A'9$, $3B'3$, $1A6C'5C'8B11C1$, $2C12A'10B7A4B'2$.  After gluing all sides except $C'$, we get the planar model of the Girl's surface   in Fig. \ref{bg} on the right.

 \begin{figure}[ht]
\center{\includegraphics[height=5.0cm]{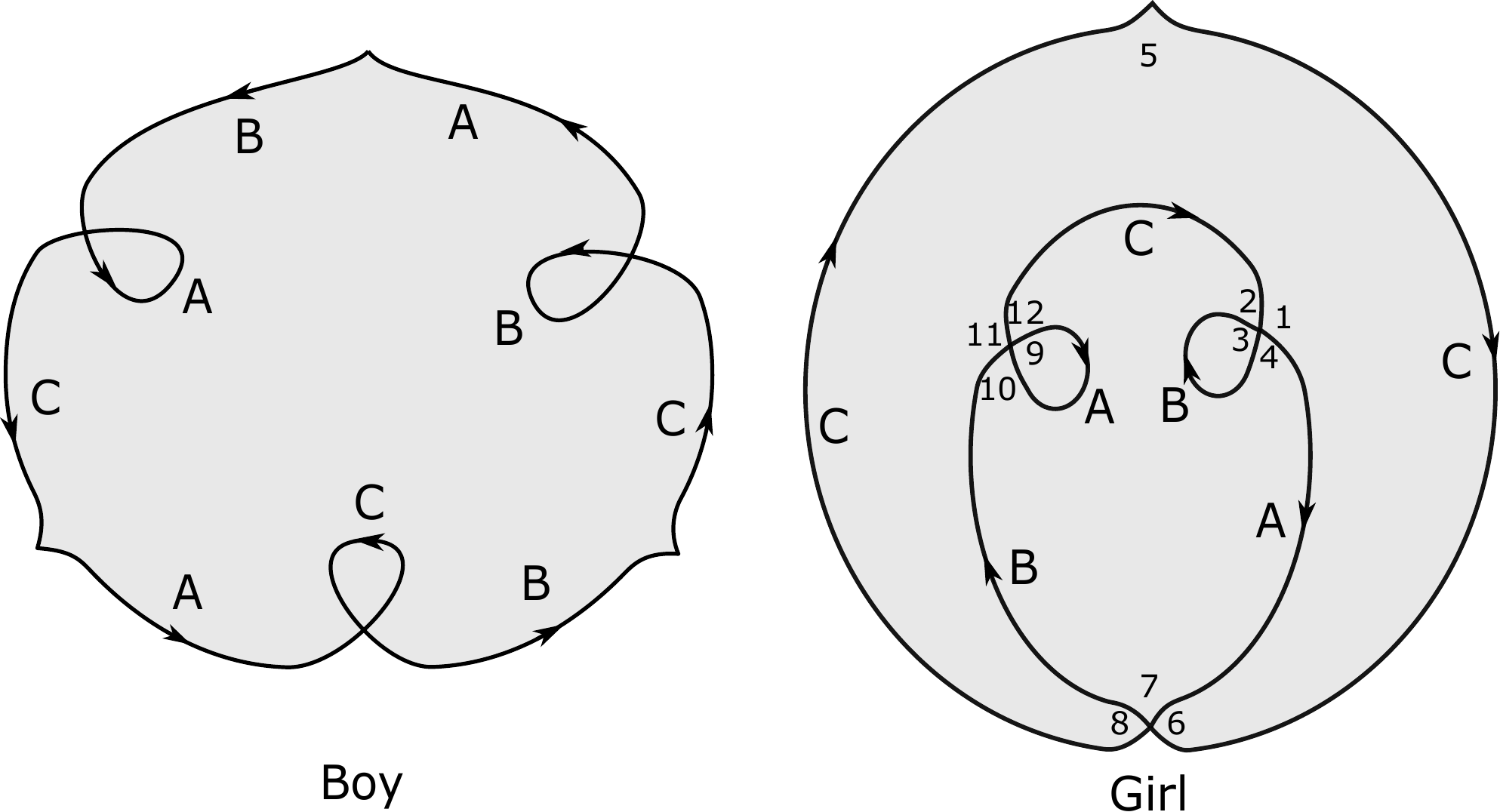}}

\caption{Planar model of the Boy's and Girl's surfaces}

\label{bg}
\end{figure}

\begin{lemma} The Fig. \ref{bg} gives unique (up to homeomorphism) representations of projective plane in the form of the 2-disk with glued opposite  boundary points and such that the boundary is mapped to the 1-skeleton of the Boy`s and Girl`s surfaces, respectively.
\end{lemma}
\begin{proof}
If we glue the other sides of the first planar model then we obtain the M\"obius band that cannot be placed on the plane. In the second model, any other gluing leads to gluing of $C'$ and then it neighborhood is homeomorphic to M\"obius band. 
\end{proof}

We introduce next notion for regions on the Girl's surface: 

LD (Left Disk) is a region, which contains 9 ($9A'9$-cell),

RD (Right Disk) contains 3 ($3B'3$-cell),

BR (Boundary Region) contains 5 ($1A6C'5$...-cell),

CR (Central Region) contains 7 ($7A4B'2$...-cell).

CW structures on Boy's and Girl's surfaces form their stratifications, where the stratas correspond to the cells. Also, it set a CW-structures on $S^3$ with two 3-cells and the Boy's (Girl's) surface as 2-skeleton. We consider it in detail in section 6.

\section{Flows with a single fixed point}

By analogy with the Boy`s surface, the null-point is a fixed point for a flow on the Girl's surface \cite{Dibeo-prish}. 1-stratas (1-cells) consist of fixed points and flow trajectories. A flow cannot have closed trajectories or oriented cycles, because if such exist, there will be a second fixed point inside them. So, each trajectory begins and ends at the null-point. Cutting the surface by 1-cells and separatrices, we get regions that can be of two types: 1) elliptical - all its trajectories begin and end in one corner (loops) or 2) polar - trajectories begin and end in different corners.

We don't consider the trajectories the 1-stratas as  separatrices.

Our aim is to find all the flows (up to topological equivalence) without separatrices. Since axial symmetry with respect to the vertical axis is a topological equivalence, each flow with a direction of $C$ from 2 to 12 is equivalent to a flow with a direction of $C$ from 12 to 2. Therefore,  we fix one direction of $C$: from 12 to 2.

By $-A$ and $-B$ we denote orientations that is inverse to a given orientation $A$ and $B$ as on Fig. \ref{bg} right. 
There are 4 possible flows with different orientation of 1 strata $A$ and $B$: 

1) With given orientations $A$ and $B$,

2) With orientations $- A$ and $B$,

3) With orientations $A$ and $- B$, 

4) With orientations $- A$ and $- B$.

In all cases, the LD and RD  are elliptical regions. 

In the first case, the BR has the angle  8 as the source and the angle 6 as a sink and it is therefore polar. However, the CR region has two sources -- angles 4 and 12, as well as two sinks -- angles 2 and 10, so there is at least one separatrix in it, that connect 4 and 12 or 2 and 10. 

 In the second case, in the BR region there is a source angle 8 and a sink 1, in CR -- 7  is a source, and 2 is a sink. So both regions are polar and there is only one structure of such a flow.

 In the third case, 11 is a source, and 6 is a sink in the BR, 12 is a source, 7 is a sink in the CR.  So, there is one flow structure. 

In the fourth case, 11 is a source and 1 is a sink in the BR, 10 is a source, 4 is a sink in the CR.  So, there is one flow structure.

\begin{figure}[ht]
\center{\includegraphics[height=5.2cm]{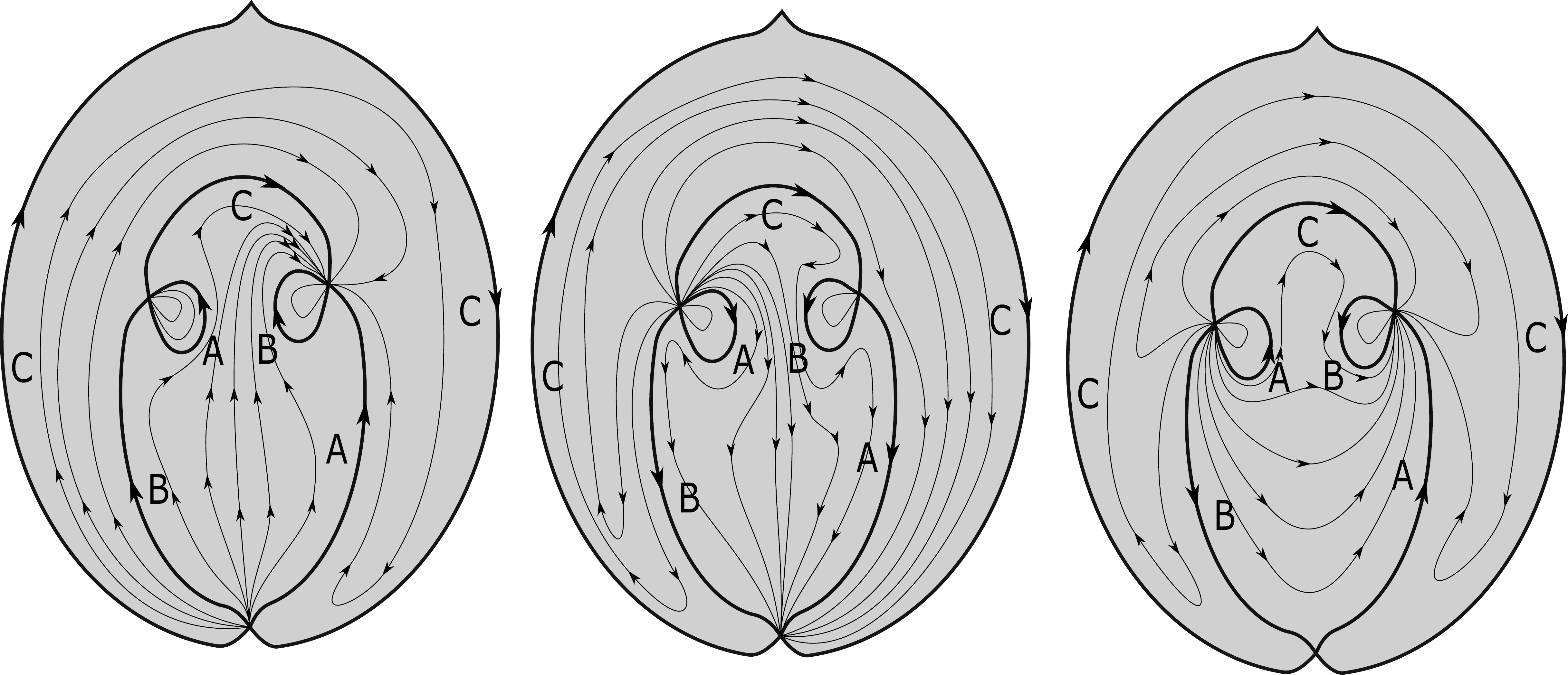}}
\caption{Flows 
with one fixed points
}
\label{gl1}
\end{figure}

To sum up, we have.

\begin{theorem}  There are 3 different flow structures with one fixed point and no internal separatrices  on the Girl's surface (see Fig. \ref{gl1}).
\end{theorem}

\section{Structural stability of Morse-Smale flows}
To prove structural stability of Morse-Smale flows on the Boy's and Girl's surfaces we need the following lemma: 
\begin{lemma}
 Let the angle n in the neighborhood of null-point is given as $x \ge 0, y \ge 0$ in the corresponding coordinate system $(x,y)$.  Then Morse-Smale vector field is the angle n has a form  $\{x \ f(x,y),\ y \  g(x,y)\}$ where $f(0,0)\ne 0$ and $g(0,0)\ne 0$.
\end{lemma}
\begin{proof}
Let the field look like $\{X(x,y),Y(x,y) \}$. Condition for touching the field to the axis $x$  can be recorded as $Y(x,0)=0$ and therefore $Y(x,y)= y \ g(x,y)$. Similarly, $X(x,y)= x \ f(x,y)$ . Then   
$$\frac{\partial X}{\partial x} (0,0)= f(0,0), \ \
\frac{\partial X}{\partial y} (0,0)= 0, \ \ 
\frac{\partial Y}{\partial x} (0,0)= 0, \ \ 
\frac{\partial Y}{\partial y} (0,0)= g(0,0)
$$  and the condition of non-degeneration of the singular point is written as $f(0,0)\ne 0$ and $g(0,0)\ne 0$.
\end{proof}

\begin{theorem}

Morse-Smale flows and only they are structurally stable flows on the Boy's and Girl`s surfaces.
\end{theorem}
\begin{proof}
It is necessary to show that there doesn't exist a small bifurcation of Morse-Smale flow and that it exist  for other flows. If bifurcation does not occur at an angular point (0-stratum), then it is the same as on a surface with a boundary. Therefore, we can take advantage of the structural stability of Morse-Smale flows on a surface with a boundary \cite{Palis68}.  Consider a flow in a angle  with a vertex in the 0-strata. For any flow in the angle we have a representation a form  $\{x \ f(x,y),\ y \  g(x,y)\}$  as int the Lemma. With a small change of  the flow, the condition $f(0,0)\ne 0$ and $g(0,0)\ne 0$  does not change. Then the Morse-Smale flows are structural stable in 0-stratum. Using the family of flows $\{x \ (c+ f(x,y)),\ y \ (c+g(x,y))\}$, with small changes to the  parameter $c$, any flow can be brought in the neighborhood of 0 to the Morse-Smale flow. It remains to prove that there are no other flows, with the same structure as Morse-Smale flows and without bifurcations. Let the flow  $\{x \ f(x,y),\ y \  g(x,y)\}$ has $f(0,0)=0$  and $(0,0)$ is an isolated fixed point like for the Morse-Smale flows. Consider the family $$\{x \ (c^2-cx + f(x,y)-f(c,y)),\ y \ g(x,y)\},\ c\in [0,1]$$. Then with any $c>0$ the resulting flow has 2 fixed points $(0,0)$ and $c,0$. Therefore, it isn't topologically equivalent to the original flow (at $c=0$). If $g(0,0)=0$, then the bifurcation formula is similar: 
$$\{x \ f(x,y), y\ (c^2-c y + g(x,y)-f(x,c))\},\ c\in [0,1]$$.
\end{proof}

\section{Optimal Morse-Smale flows}

As with the Boy's surface \cite{Dibeo-prish}, each Morse-Smale flow has at least one fixed on every 1-cell. We call fixed point by a potential source (sink) if it is $\alpha$-limit ($\omega$-limit) point  for two trajectories belonging 1-cell.
We indicate potential sink in red and the letter R.  Potential sources are green (G). Separatrices have the same color as it limit points.

\begin{theorem}
On the Girl's surface, optimal Morse-Smale flows have 4 fixed points: 0-cell and by one point on each 1-cell. In total, there are 534 non-homeomorphic and 1058 non-homotopic optimal Morse-Smale flows.

\end{theorem}
\begin{proof}

Since $c, d$  and $g$ are glued together at the same point, they have the same color. Let it be green. Similarly, points $b$ and $f$ are painted with one color, as well as one color of points $a$ and $e$. Therefore, three options are possible: 1) $a$ is red and $b$ is green, 2) $a$ and $b$ are red, 3) $a$ and $b$ are green.

\begin{figure}[ht]
\center{ \includegraphics[height=5.0cm]{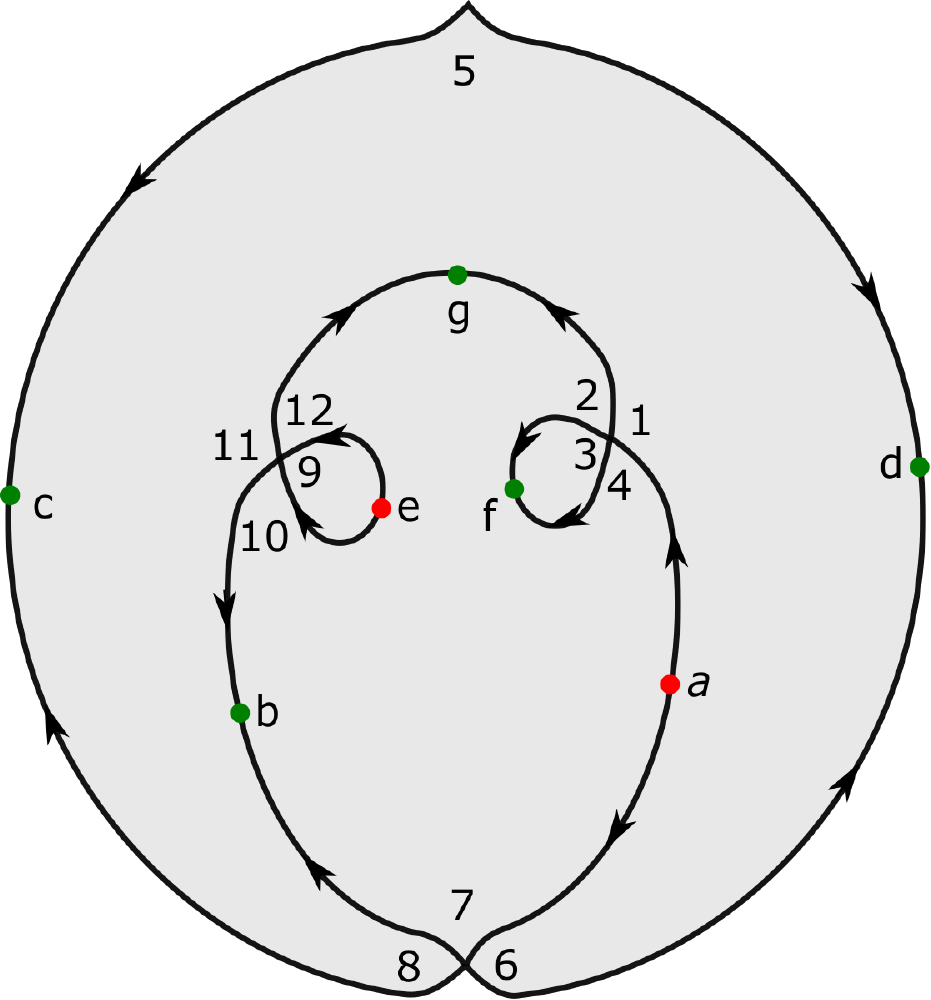} }
\caption{Morse-Smale flows, option 1)}
\label{bg3a}
\end{figure}

In  \textit{the first option} ($a$ is red and $b$ is green, see Figure \ref{bg3a}),t to calculate the total number of flow structures  $n$, find the number of flow structures $n_b$ in $BR$ and the number of flow structures  $n_c$ in $CR$.  Then $n=n_b \times n_c$.  
If there are no red separators in $BR$, then any of the points $b,c,d,g$ can be a source. So we have 4 different flow structures in $BR$ in this case. If there is a red separatrix, it is alone and ends in $a$.  If it starts at angle $8$, then it splits $BR$ into two regions in each of which have two green points: $b,g$ for the inner region and $c,d$ for the outer region. Thus, in each of these areas there can be two sinks and the total number of flows with a separatrix  $8\to a$ is $2\times 2 = 4$. Separatrices $11\to a$ and $5\to a$ split $BR$ into two region, one with a green point and the other with three green points. Then, the total number of options for each of them is $1 \times 3 = 3$. Adding everything together, we get that $n_b= 4 + 4 + 3 + 3 = 14$.

If there are no red separators in  $CR$, then the three green points on the boundary give 3 options. If there is one separatrist $2\to e$,$2\to a$, $a\to e$ or $e\to a$, then it splits $CR$ into two regions with one and two green points. Thus, we have $4\times 2= 8$ variants of flows with one separatrix. If there are two red separatrices in $CR$ ($2\to a$, $2\to e$), then only one such flow is possible. In total, we have $n_c=3+8+1=12$.  The total number of non-homeomorphic flows for the first diagram is $n= 14\times 12=168$.

In the second and third options, it is necessary to take into account cases that are symmetric with respect to the vertical axis of symmetry. We denote by bs the number of symmetric flows in $BR$ and by $b_n$ -- not symmetrical ones.  In $CR$, we denote corresponden number by $c_s$ and $c_n$, respectively. Then the total number of flow structures can be calculated by the formula $$n = b_s \times c_s + b_s \times c_n + b_n \times c_s + 2 b_n \times c_n.$$

\begin{figure}[ht]
\center{ 
\includegraphics[height=5.0cm]{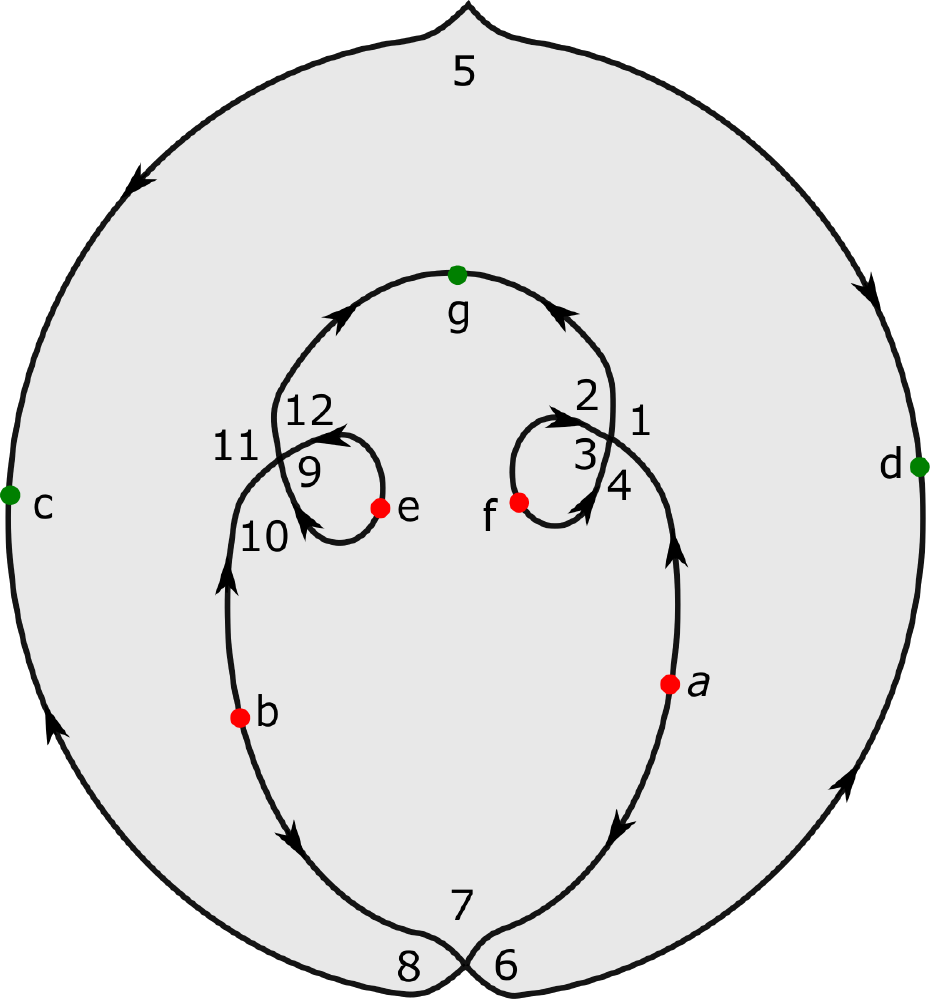}  \ \ \ \ 
\includegraphics[height=5.0cm]{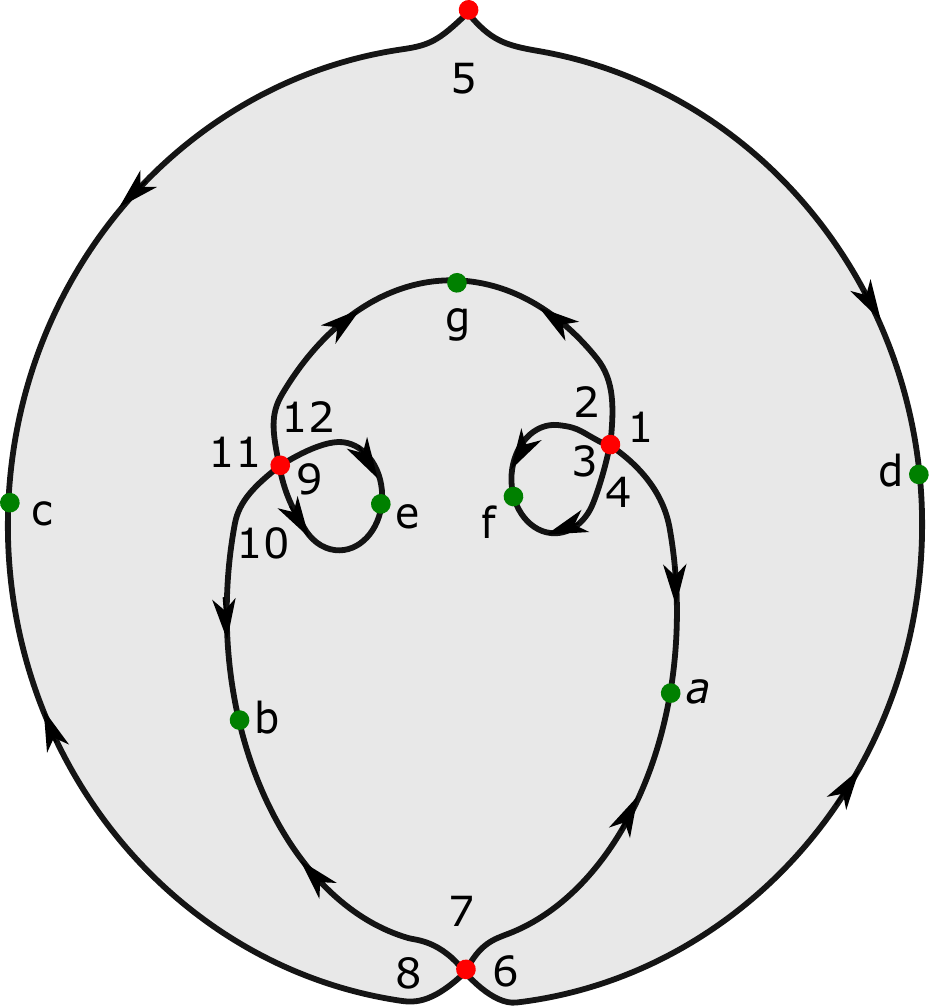}}
\caption{Morse-Smale flows, options 2) and 3)}
\label{bg3b}
\end{figure}

Consider  \textit{the second option} ( $a$ and $b$ are red, see Figure \ref{bg3b}, left). If there are no red separatrix  in $BR$, then the version with a sink  in $g$ is symmetrical, and versions with a sink in $c$ or $d$  are not symmetrical. The last two version are homeomorphic to each other. If there is only one red separator, it may end in $a$ or $b$. These cases are symmetrical (homeomorphic), so we consider only the first of them. Then there are two case of separatrix $5\to a$ or $b\to a$, which divide $BR$ into two regions with one and two possible options for a sink. So, we have 4 non-symmetrical flows structures. If there are two red separators in $BR$ ($5\to a$, $5\to b$), then the such flow is unique and symmetrical. Summing up, we get $b_s=2, b_n=5$.

In $CR$, the symmetric flow must contain the green separatrix $g\to 7$.  Then the sources can be in 1) $a$ and $b$ or 2) $e$ and $f$. Thus, $c_s=2$.  Another option for arranging sources (e.g., $a,e$) for this saparatrix is non-symmetrical. For the separatrix $g\to 4$ we have three versions of the sources $(a,b,e)$, and if there are no separators, then two non-homeomorphic versions of the source $(b,e)$.  Summing up, we get $c_n=6$. 

Then the total number of non-homeomorphic threads for the second options is $n=2\times 2+2\times 6+5\times2+ 2\times 5\times 6=86$.

For \textit{the third option}  ($a$ and $b$ are green, see Figure \ref{bg3b}, right), symmetric flows for each of the two regions should have a sink in $g$, and for non-symmetric ones, there are two non-homeomorphic possibilities for the location of the sink. Therefore, $b_s=1$, $b_n=2$, $c_s=1$,$c_n=2$.  So, $n=1\times 1+2\times 1+1\times 2+2\times 2\times 2=13$.

Summing up all three options, we have $168 + 86 + 13 = 267$ flows with green $c, d$ and the same number with red, that is, only $534$ non-homeomorphic ms-flows. 
\end{proof}

\section{Projective Morse-Smale  flows}

\begin{theorem}
An optimal projective Morse-Smale (OPMS) flow of the Girl's surface has three sources, three sinks and five saddles on the projective plane.

There are 118 non-homeomorphic OPMS flows and 230 non-isotopic OPMS flow on the Girl's surface. 

\end{theorem}
\begin{proof}

Consider  the first case where 0-cell is the source for all three corresponding vertices on the projective plane. Three options are possible for $e$ and $f$: 1) they are both sinks, 2) $e$ is a sink, $f$ is a saddle, and 3) they are both saddles.

 \begin{figure}[ht]
\center{\includegraphics[height=5.0cm]{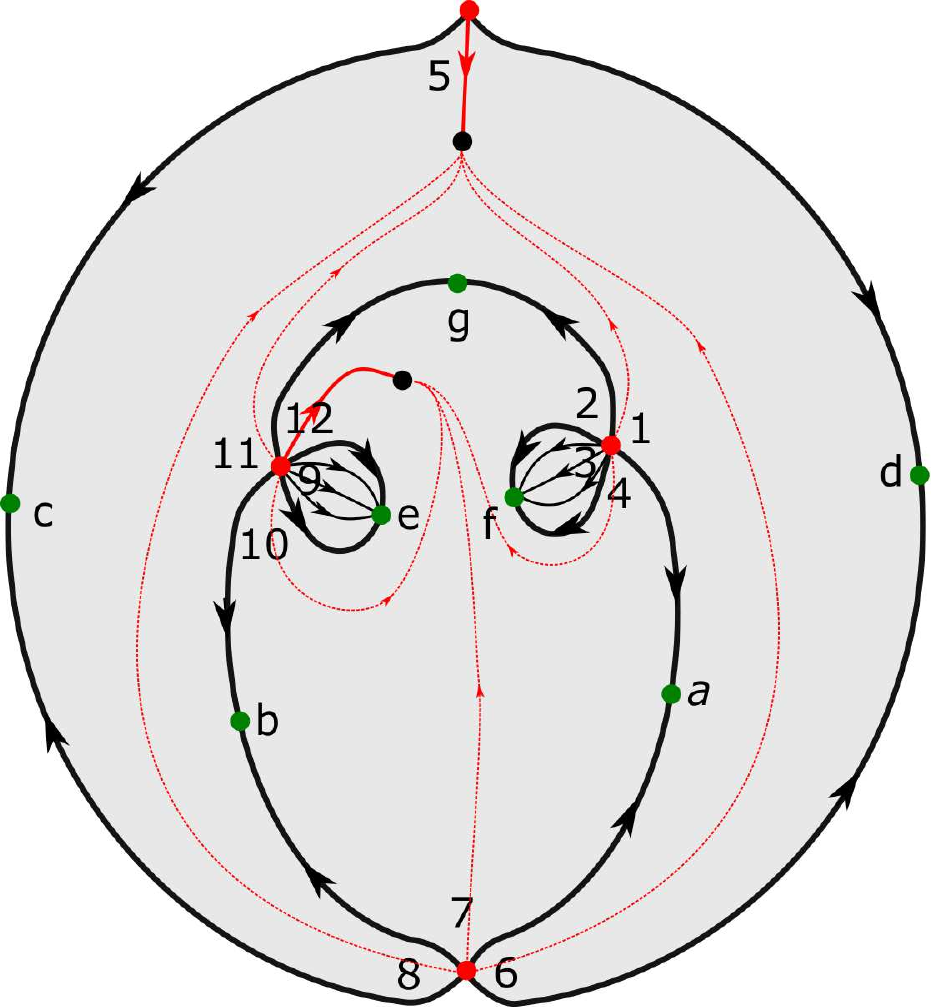} \ \ \ \ 
%
\includegraphics[height=5.0cm]{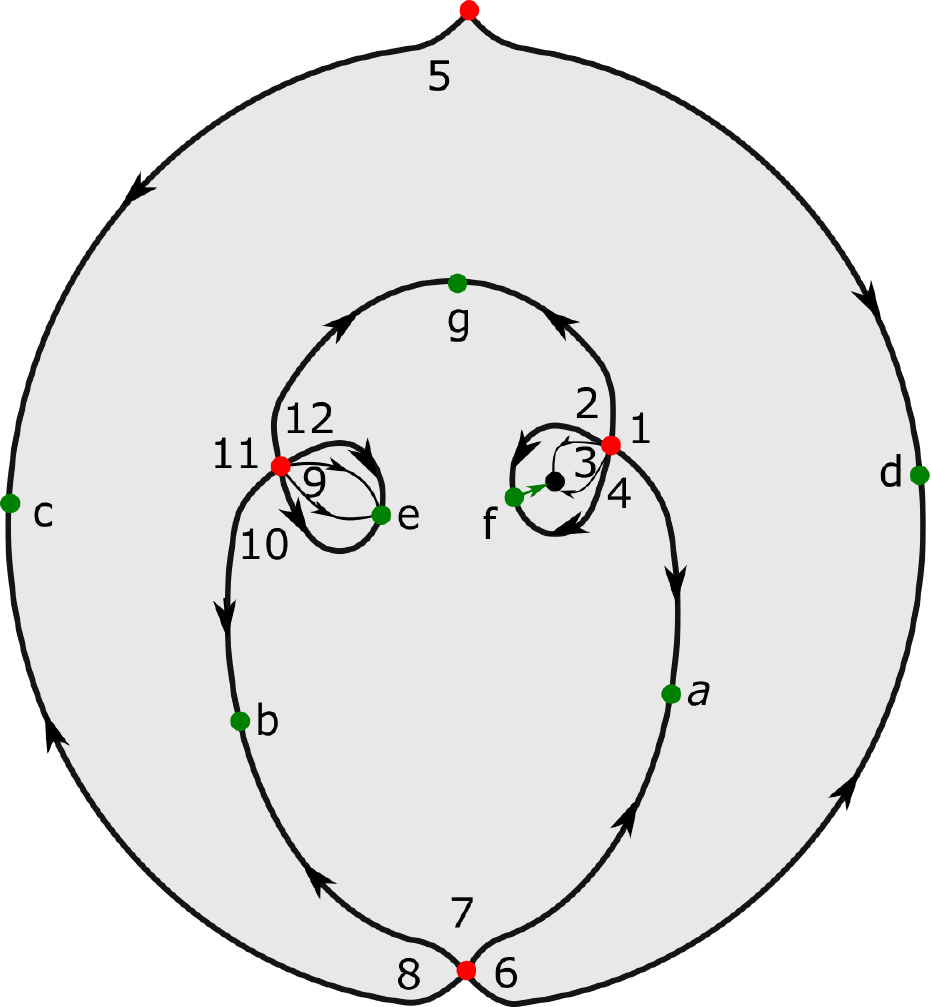}}
\caption{Projective Morse-Smale flows, options 1),2)}
\label{pms2}
\end{figure}

Consider \textit{the first option} (Fig.\ref{pms2}, left ). Since two sinks are known ($e$ and $f$), the third sink lies within or at the boundary of the $BR$.  In the case (a) of the internal sink, we have one symmetrical structure on $BR$.  The other two sinks must be separated by a stable stable manifold in the $CR$.  Let one of his separatrices come out of 12, then for the other there are three options: 4, 7, 10. The options when the first of the separators comes out of 2 are homeomorphic to those considered. So there are three non-homeomorphic flow structures in this case. 

Case (b): $c=d$ are sinks. Since they must be separated by a stable manifold of saddle in $BR$, one of the separatorix included in it begins at 5, and for the other there are 4 options. There are three options for $CR$, as in the previous case. So, for the case of b) we have $3\times 4=12$ non-homeomorphic flows.

Case (c): $g$ is sink. In this case, one separatrix begins at 12 and the other at 2. Such non-homomorphic pairs are possible, specifying two saddles inside the $CR$: 1) 2--4, 10--12 (symmetrical flow), 2) 2--7, 7--12 (symmetrical flow), 3) 2--4, 7--12, 4) 2--4, 4--12, 5) 2--12, 2--4, 6) 2--12, 2--7, 7) 2--12, 2--10. So we have 7 flow structures.

Case (d): $a$ is a sink. The following options are possible: 1 -- 6) 4--7, 2 or 12 -- 4, 7 or 10 ; 7 -- 10) 4--10, 4 or 10 -- 2 or 12, 11) 4--12, 7--12, 12) 4--12, 10--12, 13 -- 16) 2--4, 2 or 12 -- 7 or 10. We have 16 options. Total, there are $3+12+7+16=38$ flow structures (2 of them are symmetrical).

Consider \textit{the second option} $e$ is a sink, f is a saddle (Fig.\ref{pms2}, right).  The situation is similar to the first option. We consider the cases where a sink in BR is? There are 5 cases: a) internal sink -- one flow structures, b) $c = d$ are sinks -- 4 flow structures (inner saddle in BR, one separtrix enters it from 5, and for the second point there are 4 options), c) $g$ is a sink -- 4 flow sructures (inner saddle in $CR$, one separtrix enters it from 12, and for the second point there are 4 options) d) $b$ is a sink, 4 cases (10 -- 2, 4, 7 or 12), e) $a$ is a sink, 2x3=6 cases (7 or 10 -- 2, 4 or 12). Total, there is $1+4+4+4+6=19$ flow structures.

 \begin{figure}[ht]
\center{\includegraphics[height=5.0cm]{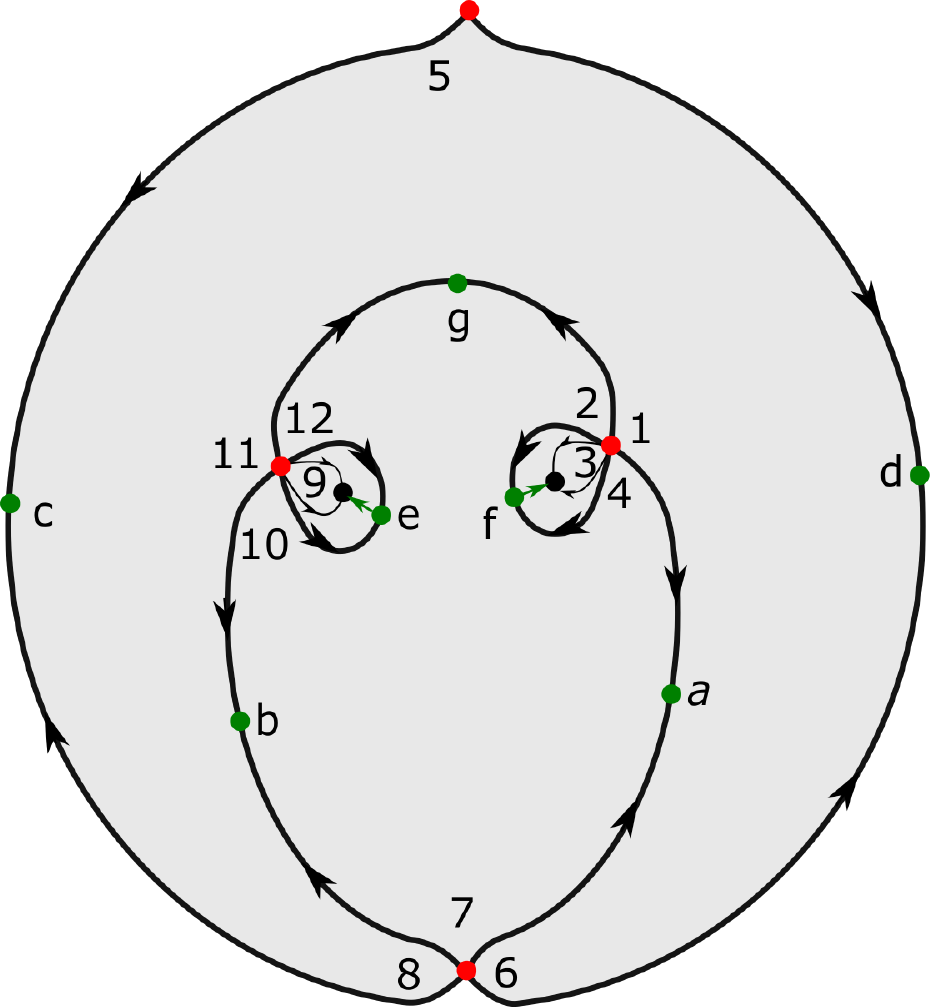}}
\caption{Projective Morse-Smale flows, option 3) }
\label{pms3}
\end{figure} 

In \textit{the third option} ($e,f$ are saddles, Fig. \ref{pms3}) the sink lie on the boundary of BR and CR.  There are two possible cases of the location of the sink: points $a$ and $g$.  In each case, the flow is unique  in each of the region. Therefore, there are 2 flow structures (the first of which is symmetrical).

Summing up, there are $ 2 \times (38 + 19 + 2) = 118$ the  projective non-homeomorphic flows. 
\end{proof}
 
\section{Homotopy of flows}
Let $V$ is a subset of the space of all smooth vector fields (flows) on $M$. A homotopy of $X \in V$ is a continuous map $h: [0,1] \to V$ such that $h(0)=X$. Then we say that $X$ is homotopic to $h(1)$ in $V$. 

If there is an isotopy $H_t$ of surface $M$ then it set homotopy $H_t^*(X)$ of $X$, $t\in[0,1]$.

\begin{lemma}
If $V$ is a sat of Morse-Smale or projective Morse-Smale vector fields and vector field $X$ is homotopic to $Y$ in $V$ then there is a continuous deformation $H: M \times [0,1] \to M, H_t(x)=H(x,t)$ such that $H_0$ is an identity mapping and $H_1$ is a topological equivalence of $X$ to $Y$.
\end{lemma}
\begin{proof}
Necessity. Using structure stability of Morse-Smale vector filed for any $t \in [0,1]$ there exist a neighborhood $U\subset [0,1]$ and topological equivalences $g_s:M \to M$  of $h(s)$ to $h(t)$, $s\in U$ and $g_s$ continuously depend of $s$. 
Due to the compactness of the segment, it is possible to cover it with such a finite number of intervals where this property is satisfied. This allows us to construct a continuous family of homotopy equivalences $g_t$ from $h(0)$ to $h(1)$. Then the desired isotopy is given by the formula $H(x,t)=g_t(x)$.  

\end{proof}
 
\textbf{Remark} The last lemma holds for a set of vector fields on the Boy's or Girl's surface with one fixed point and a given number of separatrices. For the proof, it suffices to note that the beginning and end of the sepratrix is a fixed point and it does not change under homotopy, and hence the sepratrix diagrams is homeomorphic and the vector fields are topologically equivalent. 

\begin{theorem}
Let $V$ is one of the set of vector fields (flows), which we considered above, $n$ is the total number of topological structures from $V$ and $n_s$ is the number of symmetric structures from $V$, then the number of non-homotopic structures  is given by the formula
$$m=2 \times n - n_s$$
on the Girl's surface and by formula 
$$m=3 \times(2 \times n - n_s)$$
for optimal flows on the Boy's surface.
\end{theorem}
\begin{proof}
Recall that a continuous deformation of an arbitrary orientation-preserving mapping $f$ of the segment $[0,1]$ to the identity mapping can be given by the formula $H(x,t)=tx+(1-t)f(x)$.
For a 2-disc, also every orientation-preserving homeomorphism can be continuously deformed to the identity. In addition, this deformation can be made so that the trajectories are displayed in the trajectory and then along each trajectory so as to obtain a given topological equivalence. If the diagram is not symmetrical, then symmetry about the vertical axis will result in a non-homotopic flow. Therefore, the total number of threads must be multiplied by two. In this case, we counted the symmetric flows twice. Because of this, the number of symmetric threads must be subtracted from the total number.
Taking into account that there are no $Z_3$ invariant flows among the optimal flows on the Boy's surface, then the total number of flow must be multiplied be three.
\end{proof}

 Using the formulas in the last theorem we calculate the number of  non-homotopic structures. The resulting number of structures are given in the table 6.1. 
\begin{table}[ht]
	\centering
		\begin{tabular} {|c|c|c|c|}
		\hline
surface
& 
1 fixed point
& 
Morse-Smale
& 
projective 

\\ 
\hline

Boy's 
 &
18/108
 &
342/2004
 &
80/438

 \\
\hline

Girl's 
 &
3/6
 &
534/1058
 &
118/230

\\
 \hline		

		\end{tabular}
	\caption{Number of optimal flows structures up to homeomorphism/
homotopy}
	\label{tab:NF}
\end{table}







\end{document}